\documentclass{article}

\usepackage{amsfonts}
\usepackage{amssymb}
\usepackage{amsmath}

\usepackage{hyperref}
\newcommand{\arxiv}[1]{\href{http://www.arXiv.org/abs/#1}{arXiv:#1}}
\usepackage{tikz}

\newtheorem{theorem}{Theorem}

\newenvironment{proof}[1][Proof]{\noindent\textbf{#1.} }{\ \rule{0.5em}{0.5em}}

\def\CRIT{{\mathcal G}^c}
\def\crit{{\mathcal G}^c}

\def\spann{\operatorname{span}}

\def\R{{\mathbb R}}
\def\C{{\mathbb C}}
\def\Rp{\R_+}
\def\diag{\operatorname{diag}}

\def\digr{{\mathcal G}}

\def\bH{{\mathbb H}}

\def\Sat{\operatorname{Sat}}

\def\cL{{\mathcal L}}

\begin{document}

\title{Fiedler-Pt\'{a}k scaling in max algebra\thanks{This work is partially supported by
RFBR-CRNF grant 11-01-93106}}

\author{Serge\u{\i} Sergeev\thanks{INRIA and CMAP Ecole Polytechnique,
91128 Palaiseau Cedex, France. Email: sergiej@gmail.com}}
\date{}
\maketitle

\begin{abstract}
This is essentially the text of my talk on Fiedler-Pt\'{a}k scaling in max algebra
delivered in the invited minisymposium in honor of Miroslav Fiedler at the
17th ILAS Conference in Braunschweig, Germany.\\
 


\end{abstract}


{\bf 0. Introduction}  

The previous talks brought to our attention many ideas of Miroslav Fiedler that served as gold mines in 
mathematics. This talk will be based on only one gold mine: the Fiedler-Pt\'ak diagonal similarity
scaling and the role of this scaling in max algebra. 

We start from the initial appearence of this diagonal similarity scaling in a work of Fiedler-Pt\'ak\cite{FP-67},
which was not meant to be applied in max algebra (an area which was only starting to emerge at the time of~\cite{FP-67}).

We proceed with an introduction into max algebra and a specific version of the Fiedler-Pt\'ak scaling appearing there,
namely the max eigenvector scaling.

Finally we concentrate on applications of the Fiedler-Pt\'ak scaling in max algebra, in particular,
to commuting matrices and matrix powers. In the following we abbreviate this to FP scaling. 

{\bf 1. FP scaling} Let $\digr(A)$ be a finite digraph weighted by $a_{ij}\in\Rp$.
Then, the FP scaling means existence of a potential $y$ with positive coordinates $y_i$ such that
$y_i^{-1}a_{ij}y_j\leq 1$, which occurs if and only if $\prod\limits_{(i,j)\in C} a_{ij}\leq 1$ for any
cycle $C$ in $\digr(A)$. In the {\bf strong FP scaling} the non-strict inequalities are replaced by the strict ones.

Fiedler and Pt\'ak used this scaling to characterize Hadamard preservers of certain classes of real-valued matrices.
For a real-valued matrix $A$, define $H(A)$ by
\begin{equation*}
h_{ii}=|a_{ii}|,\quad h_{ij}=-|a_{ij}|\ \text{for $i\neq j$}.
\end{equation*}
Then class $\bH$ contains all matrices $A$ such that $H(A)$ has all principal minors positive,
and class $\bH_*$, in addition, requires that the matrices have all non-zero diagonal
entries. Here is one of the equivalence theorems (rather, a part of it) characterizing
the Hadamard preservers of $\bH$ and $\bH_*$.

\begin{theorem}[Fiedler, Pt\'{a}k\cite{FP-67}, Theorem 4.1]
For $B=(b_{ij})\in\R^{n\times n}$ (where $n\geq 2$), the following are equivalent:
\begin{itemize}
\item[1.] $\forall k_1,\ldots,k_s$, $|\prod_{i=1}^s b_{k_ik_{i+1}}|\leq |\prod_{i=1}^s b_{k_ik_i}|\neq 0$;
\item[2.] $\exists$ diagonal positive $D$ s.t. $C=D^{-1}BD$ satisfying
$0\neq |c_{ii}|\geq |c_{ij}|$ $\forall i,j\in\{1,\ldots,n\}$;
\item[3.] $\forall A\in\bH$, $A\circ B\in\bH$;
\item[4.] $\forall A\in\bH_*$, $A\circ B\in \bH_*$.
\end{itemize}
\end{theorem}

FP scaling is used to establish equivalence between the first two statements.

{\bf 2. Max-algebra and FP scaling} Consider nonnegative numbers $\Rp$ equipped
with the ordinary multiplication $ab$ and the new addition $a\oplus b:=\max(a,b)$.
These operations are extended to matrices and vectors (of compatible sizes) 
in the usual way: we have $(A\oplus B)_{ij} = (a_{ij}\oplus b_{ij})$,
$(A\otimes B)_{ij} = \bigoplus_{k} a_{ik}\otimes b_{kj}$ and
$A^k =\overbrace{A\otimes\cdots\otimes A}^{k}$ for $A=(a_{ij})$ and $B=(b_{ij})$.
\if{
\begin{equation*}
\begin{split}
(A\oplus B)_{ij} &:= (a_{ij}\oplus b_{ij}) \\
(A\otimes B)_{ij} &:= \bigoplus_{k} a_{ik}\otimes b_{kj}\\
A^k &:=\overbrace{A\otimes\cdots\otimes A}^{k}.
\end{split}
\end{equation*}
}\fi

There are also max-plus and min-plus versions of max algebra. The first one
is $\R\cup\{-\infty\}$ equipped with $a\otimes b:=a+b$ and $a\oplus b:=\max(a,b)$, the second one is
$\R\cup\{+\infty\}$ equipped with $a\otimes b:=a+b$ and $a\oplus b:=\min(a,b)$,
Both are isomorphic to our principal {\bf max-times} version.
 
Note that the max-algebraic matrix powers can be expressed in terms of {\bf optimal paths of fixed length}.
For a path $P=i_1\to\cdots\to i_k$, denote
the weight and the length of $P$ by
\begin{equation*}
w(P):=a_{i_1i_2}\cdot\ldots\cdot a_{i_{k-1}i_k},\ l(P)=k-1.
\end{equation*}
Then the $i,j$ entry of $A^t$ can be expressed as the greatest weight
over all paths with length $t$, connecting $i$ to $j$:
\begin{equation*}
a_{ij}^{(t)}=\max\{w(P)\colon i\to\overset{P}{\to}\to j,\ l(P)=t\}. 
\end{equation*}
Dropping the restriction on length leads to a max-algebraic analogue of
$(I-A)^{-1}$ called {\bf Kleene star}:
\begin{equation*}
A^*:=I\oplus A\oplus A^2\oplus\cdots
\end{equation*}
Thus $a_{ij}^*$ is the {\bf greatest weight
of all paths} connecting $i$ to $j$ (for $i\neq j$). These obsevations also give a hint of
applications of max-linearity in optimization and scheduling.

The Kleene star converges if and only if there are no cycles $C$ with $w(C)>1$.
In this case, $A^*=I\oplus A\oplus\cdots\oplus A^{n-1}$.

In the following we denote by $\spann_{\oplus}(A)$ the max-algebraic column span of 
$A$, i.e., the set of all max-algebraic combinations of the columns. 
In fact, all FP scalings can be described as positive vectors $x$ in $\spann_{\oplus}(A^*)$,
and all strong FP scalings as vectors $x$ in the interior of $\spann_{\oplus}(A^*)$.
For a FP scaling, take a max-linear combination of all columns of
$A^*$ with all positive coefficients, and for a strong FP scaling (possible only if $\spann_{\oplus}(A^*)$ is full-dimensional), 
take an ordinary linear combination. Actually, $\spann_{\oplus}(A^*)$ is both a max-linear space (max cone) and an ordinary convex cone, and
the interplay between these properties has been considered in~\cite{SSB}, and Joswig-Kulas~\cite{JK-10}. See also
Butkovi\v{c}~\cite{But-03} where the strong FP scaling is used in the context of strong regularity of matrices in max algebra.

Armed with this description of all FP scalings, Butkovi\v{c} and Schneider~\cite{BS-05} gave
full solutions to several matrix scaling problems. One of these problems is to find
all diagonal scalings $X$ such that $B=X^{-1}AX$ has equal column and row maxima:
$b_{ii}=\max_j b_{ij}=\max_j b_{ji}$ for all $i$.
The solution to this problem is given by all positive vectors of $\spann_{\oplus}(Q^*)$, where
$Q:=AD^{-1}\oplus D^{-1}A$ and $D=\diag(a_{11},\ldots,a_{nn})$.

Another problem is, for $A_i,B_i,C_i\in\Rp^{n\times n}$ such that $\digr(A_i)\subseteq\digr(B_i)\subseteq\digr(C_i)$, where $i=1,\ldots,m$,
to describe all $X$ such that $A_i\leq X^{-1}B_iX\leq C_i$. Informally, we have to get $B_i$ sandwiched
between $A_i$ and $C_i$ simultaneously by means of one diagonal similarity transform.
To solve the problem, take positive vectors of $\spann_{\oplus}(Q^*)$, where
$Q:=\bigoplus_{i=1}^m B_i/C_i\oplus\bigoplus_{i=1}^m A_i^T/B_i^T$.
Here, the divisions $B_i/C_i$ and $A_i^T/B_i^T$ are defined entrywise.

Max eigenvector scaling is also a kind of FP scaling. Max eigenvectors are solutions
to the spectral problem $A\otimes x=\lambda x$. The greatest eigenvalue can be found as 
the maximum cycle geometric mean
\begin{equation*}
\lambda(A)=\max_{\text{cycles } C\in\digr(A)}\left(\prod_{(i,j)\in C} a_{ij}\right)^{1/l(C)}.
\end{equation*}
Further, the {\bf critical cycles} are defined as the cycles where 
$\lambda(A)$ is attained, and the {\bf critical graph} consists of all
nodes and edges on the critical cycles. When $\lambda(A)=1$ (which can be assumed
w.l.o.g.), the space of eigenvectors associated
with $\lambda(A)$ can be explicitly described as the space (max cone) spanned
by the columns of $A^*$ whose indices belong to the critical graph of $A$.

Max eigenvector (with eigenvalue $1$) corresponds to a scaling satisfying the following
properties: 1) for all $i,j$ we have $x_i^{-1}a_{ij}x_j\leq 1$ , 2) for all $i$ there is $j$ such that
$x_i^{-1}a_{ij}x_j=1$.  (Also note that $(i,j)\in\crit(A)$ implies $x_i^{-1}a_{ij}x_j=1$ for any FP scaling.)

For any FP scaling $x$ we can define the {\bf saturation graph} $\Sat(A,x)$ consisting of the edges $(i,j)$ where
$x_i^{-1}a_{ij}x_j=1$. We obtain that 
$\crit(A)\subseteq\Sat(A,x)$ and that the cycles of any saturation graph are precisely the
critical cycles. Property 2) of the eigenvector scaling means that in the
saturation graph of an eigenvector, each node has an outgoing edge.  

Below we give a schematic view of critical graph, and saturation graph of an eigenvector scaling.

\begin{figure}
\begin{tikzpicture}[scale=0.5]
\begin{scope}[shorten >=1pt,->]

\foreach \j in {0}
{
\foreach \i in {0,4,10,14}
{\draw (\i,\j) to (\i+1,\j+1);}
\foreach \i in {2,6,12,16}
{\draw (\i,\j) to (\i-1,\j-1);}
\foreach \i in {1,5,11,15}
{\draw (\i,\j+1) to (\i+1,\j);
\draw (\i,\j-1) to (\i-1,\j);}
}

\draw(12,0) to (14,0);
\draw(17,0) to (16,0);

\draw (-1.5,0) node{$\crit$:};
\draw (8.5,0) node{$\Sat$:};

\end{scope}
\end{tikzpicture}
\caption{Critical graph and saturation graph}
\end{figure}
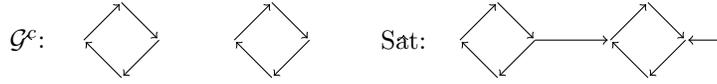

{\bf 3. Commuting matrices.} When two matrices $A$ and $B$ commute, the eigenspace of one matrix $A$ (corresponding
to some eigenvalue) is invariant
under the action of the other matrix $B$. Using this it can be proved that when $A$ and $B$ are irreducible, the
intersection of eigenspaces always contains a positive common vector~\cite{KSS}.
Take this vector $x$ and use it for
a simultaneous scaling of $A$ and $B$. Further idea is to define $A^{[1]}$ and $B^{[1]}$ as associated 
Boolean matrices of the saturation graphs $\Sat(A,x)$ and $\Sat(B,x)$, and to 
observe that matrices $A^{[1]}$ and $B^{[1]}$ also commute. This leads us to study 
{\bf commuting digraphs with nonzero outdegree of each node}. As an application,
we may obtain properties of critical graphs of commuting max-algebraic matrices (since
the strongly connected components of a critical graph are the same as 
the strongly connected components of any saturation graph).

\if{
Further if the eigenspace of $A$ is finitely generated, which means
existence of a matrix $Q$ such that it is the column span of $Q$, there exists matrix $C$ such that
$BQ=QC$. Then, choosing any eigenvector $z$ of $C$, we obtain vector $Qz$ which is in the eigenspace of $A$ and
at the same time in some eigenspace of $B$. This argument is very general, it works both in max algebra and in the 
ordinary linear algebra equally well (however, it must be assured that $Qz$ is nonzero). 

Let us restrict to the case when $A$ and $B$ are irreducible, in this case both of them have only one eigenvalue,
hence the corresponding eigenspaces have non-empty intersection. 
}\fi

Commuting digraphs can be thought of as a ``game'' where two players (Alice $A$ and Bob $B$) act together to connect one node to the other.
If one node can be reached from another using some number of actions of $A$ and $B$, then these actions can be arranged in
every possible pattern, but the connecting routes (i.e., intermediate nodes and edges) will be different in general. 
Playing with such commuting digraphs we obtained the following

\begin{theorem}[Katz, Schneider, Sergeev \cite{KSS}]
\label{commuting}
Let $\digr_1$ and $\digr_2$ commute and have nonzero outdegree of every node.
Let $N_{\mu}^1$, $\mu=1,\ldots,m_1,$ (resp.
$N_{\nu}^2$, $\nu=1,\ldots,m_2,$)
be the node sets of nontrivial s.c.c. of $\digr_1$ (resp. of $\digr_2$).
Then there exists a cycle $c_1\in\digr_1$ with all nodes in $\bigcup_{\nu} N_{\nu}^2$ (resp.
a cycle $c_2\in\digr_2$ with all nodes in $\bigcup_{\mu} N_{\mu}^1$).
\end{theorem}

\begin{proof} {\bf (in pictures)}
Here is, essentially, what we need to prove: that there is a cycle of $A$ (red) living in the s.c.c. ``houses'' of $B$ (blue
spheres). See Figure~\ref{f:result}.

\begin{figure}
\centering
\begin{tikzpicture}[scale=0.5]
\begin{scope}[blue]
\draw (0,0) circle (1.41 cm); 
\draw (2,2) circle (1.41 cm);
\draw (4,0) circle (1.41 cm); 
\draw (2,-2) circle (1.41 cm);
\end{scope}
\begin{scope}[shorten >=1pt,->,red]
\draw (0,0) to (0.67,0.67);
\draw (0.67,0.67) to (1.33,1.33);
\draw (1.33,1.33) to (2,2);
\draw (2,2) to (2.67,1.33);
\draw (2.67,1.33) to (3.33,0.67);
\draw (3.33,0.67) to (4,0);
\draw (0.67,-0.67) to (0,0) ;
\draw (1.33,-1.33) to (0.67,-0.67) ;
\draw (2,-2) to (1.33,-1.33);
\draw (2.67,-1.33) to (2,-2) ;
\draw (3.33,-0.67) to (2.67,-1.33);
\draw (4,0) to (3.33,-0.67);
\end{scope}
\end{tikzpicture}
\caption{The statement of Theorem~\ref{commuting}}
\label{f:result}
\end{figure}
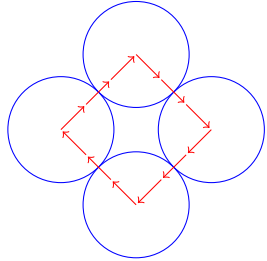

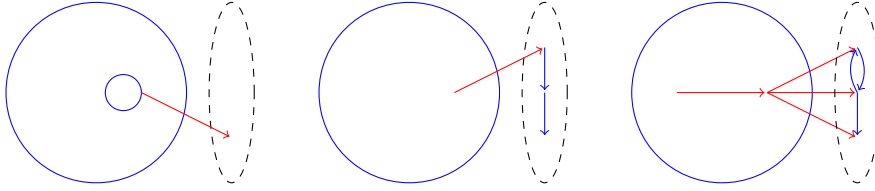
\begin{figure}
\centering
\begin{tabular}{ccccc}
\begin{tikzpicture}[scale=0.6]
\begin{scope}[blue]
\draw (0,0) circle (2 cm); 
\draw (0.6,0) circle (0.4 cm); 
\end{scope}
\begin{scope}[dashed]
\draw (3,0) ellipse (0.5cm and 2 cm); 
\end{scope}
\begin{scope}[shorten >=1pt,->,red]
\draw (1,0) to (3,-1);
\end{scope}
\end{tikzpicture}
&&
\begin{tikzpicture}[scale=0.6]
\begin{scope}[blue]
\draw (0,0) circle (2 cm); 
\end{scope}
\begin{scope}[dashed]
\draw (3,0) ellipse (0.5cm and 2 cm); 
\end{scope}
\begin{scope}[shorten >=1pt,->,red]
\draw (1,0) to (3,1);
\end{scope}
\begin{scope}[shorten >=1pt,->,blue]
\draw (3,1) to (3,0);
\draw (3,0) to (3,-1);
\end{scope}
\end{tikzpicture}
&&
\begin{tikzpicture}[scale=0.6]
\begin{scope}[blue]
\draw (0,0) circle (2 cm); 
\end{scope}
\begin{scope}[shorten >=1pt,->,red]
\draw (-1,0) to (1,0);
\draw (1,0) to (3,0);
\draw (1,0) to (3,1);
\draw (1,0) to (3,-1);
\end{scope}

\begin{scope}[dashed]
\draw (3,0) ellipse (0.5cm and 2 cm); 
\end{scope}

\begin{scope}[shorten >=1pt,->,blue]
\draw (3,1) to[bend left] (3,0);
\draw (3,0) to[bend left] (3,1);
\draw (3,0) to (3,-1);
\end{scope}
\end{tikzpicture}
\end{tabular}
\caption{The proof diagram}
\label{f:proofdiag}
\end{figure}

To show this, consider a subgraph of $A$ living in a s.c.c. of $B$. If it has cycles, then
we are done. Otherwise, we consider a node where it has to go out, and we need to show that the {\bf destination set}, i.e., 
the set of nodes to which $A$ can move after it goes out, intersects with another s.c.c. of $B$.

Now look at Figure~\ref{f:proofdiag} from left to the right (the destination set is the dashed ellipse).
We ask $B$ to turn around, since in the s.c.c. of $B$ each node lies on a cycle, 
and then allow $A$ to go out (left). But due to commutativity, we achieve the same effect if
we first allow $A$ to go out, and then ask $B$ to connect to the destination node (middle). This implies that
each node of the destination set has a predecessor in the graph of $B$, and hence there are cycles of $B$ going through
the destination set (right). Cycles of $B$ belong to the strongly connected components of $B$, which finishes the proof.
\end{proof}

{\bf 4. Matrix powers.} A theoretical motivation to study max-algebraic matrix
powers is that they are similar to powers of graphs well-known in combinatorics~\cite{BR}.
In particular, the following fact has been known for a very long time, due to
Cohen, Dubois, Quadrat et Viot~\cite{CDQV-83}.
\begin{theorem}[Cyclicity Theorem]
\label{CDQV}
Let $A\in\Rp^{n\times n}$ be irreducible 
nd $\lambda(A)=1$, then the sequence 
$A,A^2,A^3\ldots$ is ultimately periodic. That is, after some coupling time
$T(A)$ there is a period $\gamma$ such that $A^{r+\gamma}=A^r$
for all $r\geq T(A)$.
\end{theorem}
See also monographs on max algebra, such as~\cite{BCOQ,But:10,HOW:05}, for proof and
more explanation, as well as~\cite{But:10,Gav:04,HA-99,Mol-03,Nacht,BdS,SS-11} for further development.

Let us define the critical matrix $A^{[C]}$ as the associated matrix of the critical 
digraph. Then it satisfies $(A^k)^{[C]}=(A^{[C]})^k$.

Assume that $\lambda(A)=1$ and that $A$ is visualized (i.e., after an application
of FP scaling). 
Let $C$, resp. $R$
be the matrices extracted from the critical columns, resp.
rows, of $(A^{\gamma})^*$,
where $\gamma$ is the cyclicity of $\CRIT(A)$, see Figure~\ref{f:CR}. Let
$S:=A^{[C]}$ {\bf (critical matrix)}.

\begin{figure}
\centering
\begin{tikzpicture}[scale=1]
\draw (0,0) node{$(A^{\gamma})^*$}; \draw (-2.25,2.25)
node{$\crit(A)$}; \draw (-3,3) -- (3,3); \draw (-3,-3) -- (-3,3);
\draw (-3,-3) -- (3,-3); \draw (3,-3) -- (3,3);
\begin{scope}[dashed]
\draw (-2.25,2.25) circle (0.75cm);
\end{scope}
\begin{scope}[red]
\draw (3,3) -- (-3,3); \draw (3,2.5) -- (-3,2.5); \draw (3,2) --
(-3,2); \draw (3,1.5) -- (-3,1.5); \draw (-3.3,2.25) node{$R$};
\end{scope}
\begin{scope}[brown]
\draw (-3,3) -- (-3,-3); \draw (-2.5,3) -- (-2.5,-3); \draw (-2,3)
-- (-2,-3); \draw (-1.5,3) -- (-1.5,-3); \draw (-2.25,3.3)
node{$C$};
\end{scope}
\end{tikzpicture}
\caption{Extracting $C$ and $R$ terms}
\label{f:CR}
\end{figure}
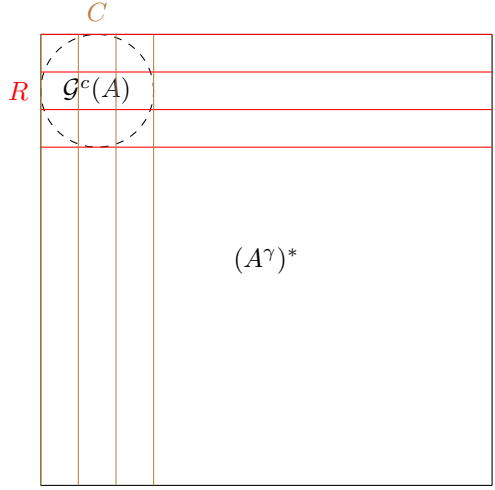

Then the following form of cyclicity theorem can be seen:
\begin{theorem}[Schneider, 2009]
\label{schneider}
Let $A\in\Rp^{n\times n}$ be irreducible and visualized,
$\lambda(A)=1$, then there exists $T(A)$
such that $A^t=C S^t R$ for all $t\geq T(A)$.
\end{theorem}

In particular we obtain that
$A^{t+\gamma}= A^t$ for all $t\geq T(A)$, 
$S^t R$ yields critical rows of $A^t$ and
$C S^t$ yields critical columns of $A^t$ for $t\geq T(A)$.

The matrix powers can be expressed in terms of optimal paths, and
the cyclicity theorem also has a path sense. Namely, one defines
{\em strong paths} as paths traversing a node of the critical graph.

\begin{theorem}[Sergeev, Schneider~\cite{SS-11}]
Let $\Pi_{ij,t}^{str}$ be the set of strong paths connecting $i$ to $j$
of length $t$ and let $w(\Pi_{ij,t}^{str})$ be the greatest weight of a path in
this set. Then $w(\Pi_{ij,t}^{str})=(CS^tR)_{ij}$  after $t\geq 3n^2$. 
\end{theorem}

We remark that the bound $3n^2$ can be reduced to $2n^2$ using the techniques
of Hartman-Arguelles~\cite{HA-99} Theorem 4. 

In the scheduling applications it is important to bound $T(A)$ from above. 
As a vague idea of doing this, let us start with a conjecture that after
$t\geq O(n^2)$, we rather have a CSR expansion of the form
\begin{equation}
\label{CSR-exp}
A^t=\bigoplus_k \lambda_k^t C_k S_k^t R_k,
\end{equation}
since at early times the matrix powers are not yet dominated by the 
critical graph and the main CSR term.

An example of CSR expansion is shown below. This expansion is called the
{\bf Nachtigall expansion}, see \cite{Mol-03,Nacht,SS-11}.

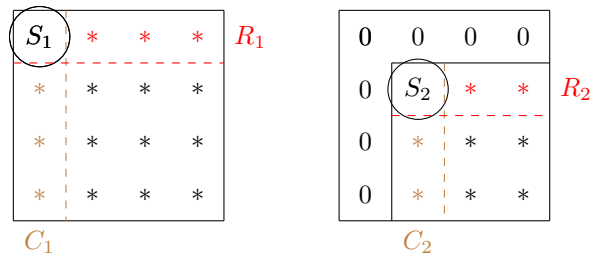
\begin{figure}
\centering
\begin{tabular}{cccc}
\begin{tikzpicture}[scale=0.7]
\draw (-2.5,2.5) node[circle,draw]{$S_1$}; 
\begin{scope}[red]
\foreach \i in
{-1.5,-0.5, 0.5} {
\draw (\i ,2.5) node{$*$};
}
\end{scope}

\begin{scope}[brown]
\foreach \i in {1.5,0.5,-0.5}
{
\draw (-2.5,\i) node{$*$};
}
\end{scope}

\foreach \i in {0.5,-0.5,-1.5}
{
\foreach \j in {1.5,0.5,-0.5}
{
\draw (\i,\j) node{$*$};
}
}

\draw (-2.5,2.5) node[circle,draw]{$S_1$};

\draw (-3,3) -- (1,3);
\draw (-3,-1) -- (-3,3); \draw (-3,-1) -- (1,-1); \draw (1,-1) --
(1,3);

\begin{scope}[dashed,brown]
\draw (-2,3) -- (-2,-1);
\draw (-2.5,-1.4) node{$C_1$};
\end{scope}

\begin{scope}[dashed,red]
\draw (-3,2) -- (1,2);
\draw (1.5,2.5) node{$R_1$};
\end{scope}

\end{tikzpicture}

&&

\begin{tikzpicture}[scale=0.7]
\draw (-2.5,2.5) node{$0$}; \foreach \i in {-2.5,-1.5,-0.5, 0.5} {
\draw (\i ,2.5) node{$0$}; }

\foreach \i in {1.5,0.5, -0.5} { \draw (-2.5,\i)
node{$0$};
}

\begin{scope}[red]
\foreach \i in {-0.5,0.5}
{
\draw (\i ,1.5) node{$*$};
}
\end{scope}

\begin{scope}[brown]
\foreach \i in {0.5,-0.5} { \draw (-1.5,\i) node{$*$};

}
\end{scope}

\foreach \i in {0.5,-0.5}
{
\foreach \j in {0.5,-0.5}
{ \draw (\i,\j) node{$*$};}}

\draw (-1.5,1.5) node[circle,draw]{$S_2$};

\draw (-3,3) -- (1,3);
\draw (-3,-1) -- (-3,3); \draw (-3,-1) -- (1,-1); \draw (1,-1) --
(1,3);
\draw (-2,-1) -- (-2,2);
\draw (-2,2) -- (1,2);

\begin{scope}[dashed,brown]
\draw (-1,2) -- (-1,-1);
\draw (-1.5,-1.4) node{$C_2$};
\end{scope}

\begin{scope}[dashed,red]
\draw (-2,1) -- (1,1);
\draw (1.5,1.5) node{$R_2$};
\end{scope}

\end{tikzpicture}
\end{tabular}
\caption{Formation of first two terms in a 
Nachtigall expansion: a schematic example. }
\label{f:nacht}
\end{figure}

Namely, after formation of the main term we cancel all rows and columns with indices in the critical graph,
and the second term is formed exactly in the same way. This procedure goes on until we get an acyclic matrix
at some stage, and leads to a CSR expansion~\eqref{CSR-exp}. A bound on $T(A)$
can be obtained using~\eqref{CSR-exp} and following the arguments of Hartman-Arguelles~\cite{HA-99} Theorem 10:
\begin{equation}
T(A)\leq 2n^2 \frac{\max\log a_{ij} - \min\log a_{ij}}{\log\lambda_1-\log\lambda_2}.
\end{equation}

The Nachtigall expansion may seem too straightforward. A different scheme can be
suggested using another kind of FP scaling called {\bf max-balancing}.

\begin{theorem}[H.~Schneider, M.~H.~Schneider~\cite{SS-91}]
For any irreducible $A\in\Rp^{n\times n}$ there is $x\in\Rp^n$ s.t.
$B=X^{-1}AX$ has the following equivalent properties:
\begin{itemize}
\item {\bf (max-balancing)} For any cut $M\subseteq\{1,\ldots,n\}$, $\overline{M}:=\{1,\ldots,n\}\backslash M$,
the greatest weight of the edges from $\overline{M}$ to $M$ equals the greatest weight of the edges
from $M$ to $\overline{M}$.
\item {\bf (cycle cover)} For each edge $(i,j)$ there exists a cycle $c$ containing this edge,
where the weight of $(i,j)$ is minimal.
\end{itemize}
\end{theorem}

An idea due to Hartman-Arguelles~\cite{HA-99} is that in a max-balanced matrix,
we can consider a threshold digraph consisting of all edges with weights not preceeding
some parameter. Lowering this parameter until the threshold digraph has a new strongly connected 
component determines the ``region of influence'' of the main CSR term, and in the same vein, of all
the rest to be formed. This idea will be made precise in a forthcoming work of the author.

{\bf 5. Other works.}

The author's works on application of FP scaling were inspired by 
the study of power method in max algebra due to Elsner and van den Driessche~\cite{ED-01}.
We also mention that the same authors use FP scaling in their study of symmetrically reciprocal
matrices~\cite{ED-04}.

The idea of FP scaling was further extended to semigroups by
Gaubert~\cite{Gau-96} who gave a positive solution to Burnside problem for the semigroups
of matrices in max algebra. In a more recent work, Benek Gursoy and Mason~\cite{GM-11b} extend the strong 
FP scaling to the set of matrices, applying it to study the asymptotic stability of max-algebraic difference inclusions.

The same idea was crucial for the development of max-algebraic methods in
the (ordinary) eigenvalue perturbation theory, aimed to improve
and extend the theory of Vi\v{s}ik, Ljusternik and Lidski\u{\i}. Namely, Akian, Bapat and Gaubert~\cite{ABG-04} consider matrices of the form
\begin{equation*}
(A_{\epsilon})_{ij}=a_{ij}\epsilon^{A_{ij}}+o(\epsilon^{A_{ij}}),\quad \epsilon\to^+ 0
\end{equation*}
where $a\in\C^{n\times n}$, $A\in(\R\cup\{+\infty\})^{n\times n}$.
One of the goals is to give the first order asymptotics of the eigenvalues
\begin{equation}
\label{asympt}
\cL_{\epsilon}^i\sim\lambda_i\epsilon^{\Lambda_i}.
\end{equation}
It can be observed that the ordinary FP scaling (of the min-plus matrix $A$, in terms of min-plus algebra) 
already helps, visualizing some entries of $a$ (corresponding to the critical graph of $A$) which are really important for
the asymptotics.
Further idea is to extend the critical graph to cover the whole $A$, which is achieved by means of a diagonal
similarity scaling similar to max-balancing, to leave out all unimportant entries of $a$ and to construct
a sequence of Schur complements in what remains of $a$. This sequence is determined by a critical graph tower associated with $A$
(in min-plus algebra). This critical graph tower also determines the exponents $\Lambda_i$ in~\eqref{asympt}, and $\lambda_i$
appear as groups of eigenvalues (in the ordinary linear algebra), determined by the sequence of Schur complements. See~\cite{ABG-04}~Theorem 5.1 for precise formulation.   

We also mention that the diagonal similarity scaling of FP type is useful in the idempotent analysis (extension to max-algebraic 
functional spaces), see Kolokoltsov and Maslov~\cite{KM:97}~Chapter~2.
 
{\bf 6. Acknowledgement.} 

I wish to thank my coauthors Peter Butkovi\v{c} and Hans Schneider for their help and support.


\end{document}